\documentclass{amsart}
\usepackage{mathtools}
\usepackage{tikz}
\usepackage{hyperref}

\theoremstyle{theorem}
\newtheorem{theorem}{Theorem}

\theoremstyle{definition}

\begin{document}

\title[Signed Sums of Non-Integer Powers]{An extension of the Prouhet-Tarry-Escott problem to non-integer powers}
\author{David Treeby}
\author{Edward Wang}

\maketitle

\begin{abstract}
We present an extension of the Prouhet-Tarry-Escott problem by demonstrating that signed sums of non-integer powers of consecutive integers can be made arbitrarily close to zero. 
\end{abstract}

\section{Introduction}
The Prouhet–Tarry–Escott problem \cite{wright,prouhet} seeks to divide a set $S$ of integers into two disjoint subsets $S_1$ and $S_2$ so that the equation \[
  \sum_{x \in S_1} x^j = \sum_{x \in S_2} x^j
\] holds true for all integers $j$ from $0$ to $k$. 
It is known (but insufficiently well-known) that this condition can be satisfied if $S$ is any sequence of $2^{k+1}$ consecutive integers. For example, if $x_1,x_2,\dots,x_{16}$ is any sequence of sixteen consecutive integers, 
\[
\arraycolsep=1.2pt\def\arraystretch{1.2}
\begin{array}{cccccccccc}
+&x_1^j&-&x_2^j&-&x_3^j&+&x_4^j\\
-&x_5^j&+&x_6^j&+&x_7^j&-&x_8^j\\
-&x_9^j&+&x_{10}^j&+&x_{11}^j&-&x_{12}^n\\
+&x_{13}^j&-&x_{14}^j&-&x_{15}^j&+&x_{16}^j&=&0
\end{array}
\]
when $j$ is equal to $0,1,2$ or $3$.
This identity leads to the following remarkable fact. Given any set of sixteen cubes whose side lengths are consecutive integers, we can partition the cubes into two sets so that in each set:
\begin{enumerate}
    \item the total number of cubes is the same,
    \item the sum of their side lengths is the same,
    \item the sum of their surface areas is the same,
    \item the sum of their volumes is the same. 
\end{enumerate}
This beautiful result follows from a simple argument that we recreate below.  Following this, we show how the result can be generalized to (certain) non-integer values of $j$. Indeed, such signed sums can be made as close to zero as desired so long as the smallest of these consecutive integers is sufficiently large. For example,
\begin{align*}
+\sqrt{2009}-\sqrt{2010}-\sqrt{2011}+\sqrt{2012}&\\
-\sqrt{2013}+\sqrt{2014}+\sqrt{2015}-\sqrt{2016}&\\
-\sqrt{2017}+\sqrt{2018}+\sqrt{2019}-\sqrt{2020}&\\
+\sqrt{2021}-\sqrt{2022}-\sqrt{2023}+\sqrt{2024}&\approx -0.000000000163.
\end{align*}

\section{Non-negative Integer powers}

The results in this note stem from one simple insight: if $f$ is a polynomial of degree $n$, then $f(x+c) - f(x)$ is a polynomial of degree $n-1$, for any $c\neq 0$. This allows us to define a sequences of polynomials with diminishing degrees. Specifically, for the non-negative integer $j$ we define \(f^j_n(x)= f^j_{n-1}(x)-f^j_{n-1}(x+2^{n-1})\)
with initial term  $f^j_0(x)=(x+1)^j$. To illustrate this, the next two terms are
\begin{align*}
    f^j_1(x)=f^j_{0}(x)-f^j_{0}(x+1)&=+\,(x+1)^j-(x+2)^j,\\
    f^j_2(x)=f^j_{1}(x)-f^j_{1}(x+2)&=+\,(x+1)^j-(x+2)^j-(x+3)^j+(x+4)^j.
\end{align*}
In general, we find that \[
  f^j_n(x)=\sum_{k=1}^{2^n} {s_k(x+k)^j}
\] where $(s_k)_{k=1}^\infty=(1,-1,-1,1,-1,1,1,-1,\dots)$ is the \emph{Thue-Morse} sequence on the integers $1$ and $-1$ \cite{Allouche}. Also note that the degrees of the polynomials in the sequence diminish:
\begin{align*}
\deg f_0^j&=j,\\
\deg f_1^j&=j-1,\\
\deg f_2^j&=j-2,\\
          &\vdotswithin{=} \\
\deg f_j^j&=0.
\end{align*}
We collect these observations in the following theorem.
\begin{theorem}\label{thm1}
The degree of $ f^j_n$ is $j-n$. In particular, $f^j_j$ is a non-zero constant polynomial, and $ f^j_{n}$ is the zero polynomial provided $j<n$. 
\end{theorem}
The result given in the introduction then follows from Theorem \ref{thm1} by letting $n=4$ and noting that $f_4^0,f_4^1,f_4^2$ and $f_4^3$ are all equal to the zero polynomial.  

\section{Non-integer powers}
The previous theorem shows that $ f^j_n(x)=0$ for all $x$ given non-negative integers $n$ and $j$ satisfying $n>j$. Our aim is to generalise Theorem \ref{thm1}, yielding a result that is applicable when $j$ is not an integer. The desired generalization is stated in Theorem \ref{thm3}. To prove this, we will require the following famous result, the proof of which can be found in most analysis textbooks \cite{knuth}. 
\begin{theorem}[Newton's generalized binomial theorem] If $j\geq 0$ is not an integer and $-1 < x < 1$, then 
\[ (1+x)^j = \sum_{i=0}^{\infty} \binom{j}{i}x^i. \]
\end{theorem}
\noindent Any proof of this theorem will essentially rest on the fact that the difference between $(1+x)^j$ and the partial sum for the series can be estimated: 
\begin{equation}\label{interestingequation}
(1+x)^j-\sum_{i=0}^{n} \binom{j}{i}x^i = \sum_{i=n+1}^{\infty} \binom{j}{i}x^i =O(x^{n+1}).
\end{equation}
\begin{theorem}\label{thm3}
If $n$ is a non-negative integer and $j\in[0,n)$, then $ f^j_n(x)\to 0$ as $x\to\infty$. 
\end{theorem}
\begin{proof}
Suppose $j\in[0,n)$ is a real number. Using the binomial theorem, we find that 
\begin{align*}
  f^j_n(x)&=\sum_{k=1}^{2^n} {s_k(x+k)^j} \\
          &=x^j\sum_{k=1}^{2^n} {s_k\left(1+\frac{k}{x}\right)^j} \\
          &=x^j\sum_{k=1}^{2^n} s_k \left[\sum_{i=0}^\infty{\binom{j}{i}\left(\frac{k}{x}\right)^i}\right].
\end{align*}
As one of the summations is finite, we can reverse the order of summation without need for caution. Moreover, $f^i_n(0)=\sum_{k=1}^{2^n} {s_k k^i}$. Therefore
\begin{align*}
  f^j_n(x)&=x^j \sum_{i=0}^\infty{\frac{1}{x^i}\binom{j}{i}\left[\sum_{k=1}^{2^n} s_k k^i\right]} \\ &=x^j \sum_{i=0}^\infty{\frac{1}{x^i}\binom{j}{i}f^i_n(0).}
\end{align*}
Crucially, Theorem \ref{thm1} ensures that $f^i_n(0)=0$ for all integers $i<n$. Using this fact, and then reversing the order of summation once more gives
\begin{align*}
  f^j_n(x)&=x^j \sum_{i=n}^\infty{\frac{1}{x^i}\binom{j}{i}f^i_n(0)} \\
          &= x^j \sum_{i=n}^\infty{\binom{j}{i}\frac{1}{x^i}\sum_{k=1}^{2^n} {s_k k^i}}\\
          &=x^j\sum_{k=1}^{2^n} s_k \left[\sum_{i=n}^\infty{\binom{j}{i}\left(\frac{k}{x}\right)^i}\right].
\end{align*}
The term in parentheses is the remainder term in a (generalized) binomial expansion as seen in  \eqref{interestingequation}. Moreover, as the outer summation is finite we conclude that 
\[
f^j_n(x)=x^j O(1/x^n)=O(x^{j-n}).
\]
Therefore, this shows that $f^j_n(x)\to 0$ as $x\to\infty$ provided $n>j$.
\end{proof}
\newpage
In Figure \ref{fig:1} we can observe the behavior of $f^j_3(x)$ for $j\in[0,3]$. 
\begin{figure}[!htbp]
  \centering
  \begin{tikzpicture}
    \node[inner sep=0pt] at (0,0)
      {\includegraphics[width=.5\textwidth]{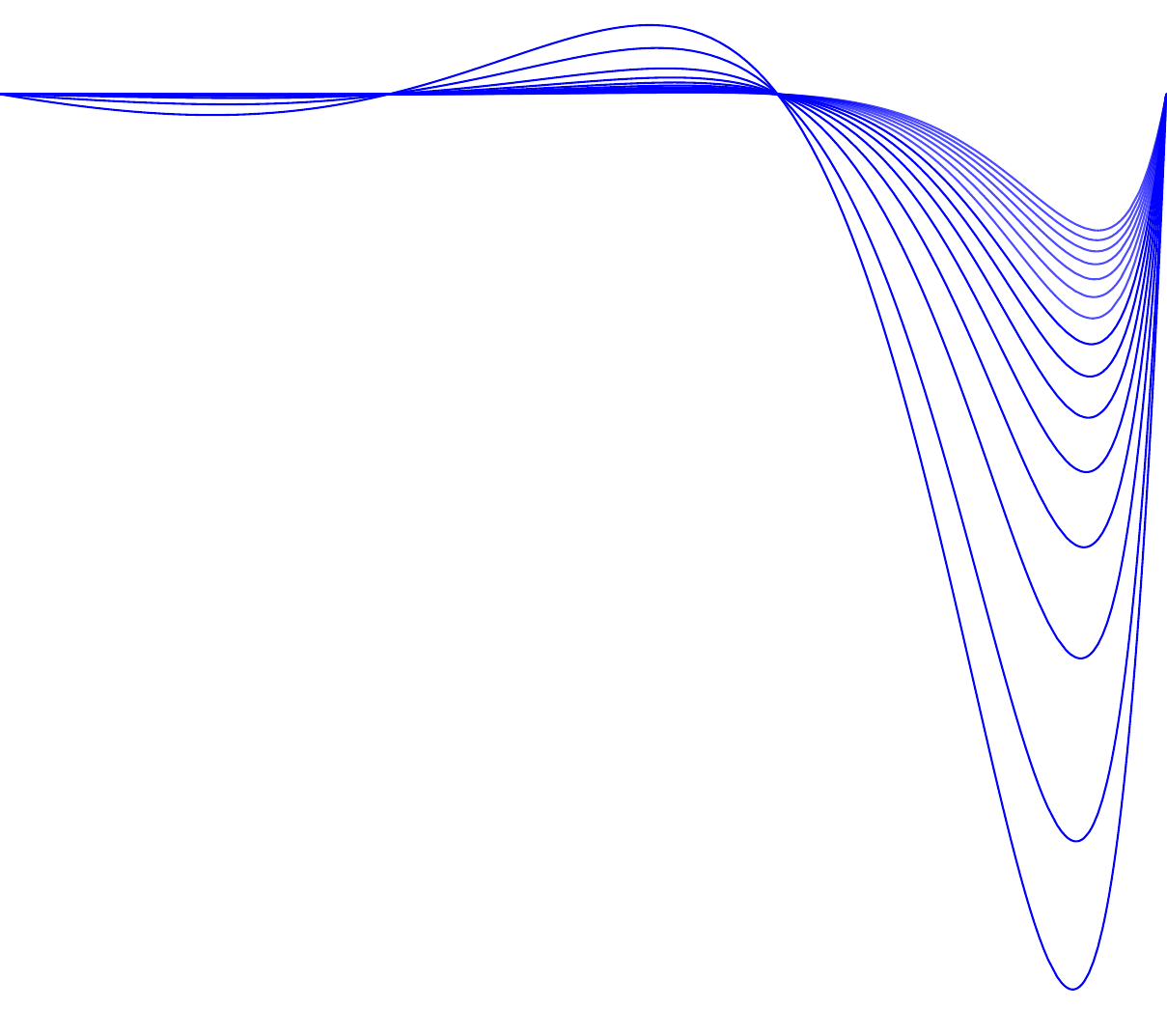}};
    \draw[->] (-2,2.3)--(3.5,2.3) node[right]{$j$};
    \fill (-3.,2.3) circle(1.75pt) node[anchor=south east]{$0$};
    \fill (-1,2.3) circle(1.75pt) node[above]{$1$};
    \fill (1.05,2.3) circle(1.75pt) node[above]{$2$};
    \fill (3.15,2.3) circle(1.75pt) node[above]{$3$};
    \draw[->] (-3,-2.8)--(-3,3)node[above]{$f^j_3(x)$};
    \node at (2.8,-2.7) {\footnotesize{ $x=1$}};
    \node at (2.8,-1.9) {\footnotesize{ $x=2$}};
    \node at (2.8,-0.9) {\footnotesize{ $x=3$}};
    \node at (2.8,-0.3) {\footnotesize{ $x=4$}};
  \end{tikzpicture}
  \caption{$f^3_j(x)\to 0$ as $x\to\infty$}
  \label{fig:1}
\end{figure}

\bibliographystyle{vancouver}
\bibliography{MonthlyReferences}

\begin{thebibliography}{1}

\bibitem{wright}
Wright EM.
\newblock Prouhet's 1851 Solution of the Tarry-Escott Problem of 1910.
\newblock The American Mathematical Monthly. 1959;66(3):199-201.

\bibitem{prouhet}
Prouhet E.
\newblock Memoire sur quelques relations entre les puissances des nombres.
\newblock C R Acad Sci. 1851;33(1):225.

\bibitem{Allouche}
Allouche JP, Shallit J.
\newblock The Ubiquitous {Prouhet-Thue-Morse} Sequence.
\newblock In: Ding C, Helleseth T, Niederreiter H, editors. Sequences and their
  Applications. Discrete Mathematics and Theoretical Computer Science. London:
  Springer; 1999. p. 1-16.
\newblock Available from: \url{https://doi.org/10.1007/978-1-4471-0551-0_1}.

\bibitem{knuth}
Graham RL, Knuth DE, Patashnik O.
\newblock Concrete mathematics : a foundation for computer science.
\newblock 2nd ed. Addison-Wesley; 1994.

\end{thebibliography}



\vfill\eject

\end{document}